\documentclass[final,3p,times]{elsarticle}

\usepackage{amssymb}

    \usepackage{amsmath}
    \usepackage{amsthm}
    \usepackage{amssymb}
    \usepackage{latexsym}

\usepackage{amssymb, enumerate}
\usepackage{graphics,graphicx,amsmath,latexsym,amssymb,amsthm,amssymb,amsfonts}
\newtheorem{thm}{Theorem}[section]

\newtheorem{defn}[thm]{Definition}
\newtheorem{rem}[thm]{Remark}

\newtheorem{exm}[thm]{Example}

\begin{document}
\title{Semiopen and semiclosed sets in fuzzy soft topological spaces}
\author{J. Mahanta}

\address{Department of Mathematics\\
         NIT, Silchar\\
         Assam, 788 010, India.\\
         mahanta.juthika@gmail.com}


\thanks{{\footnotesize Received: 20XXX; Accepted: 20XX}
}

\begin{abstract}
In this paper, we introduce semiopen and semiclosed fuzzy soft sets in fuzzy soft topological spaces.  Various properties of these sets are studied alongwith some characterizations. Further, we generalize the structures like interior and closure via semiopen and semiclosed fuzzy soft sets and study their various properties.

\end{abstract}
\maketitle
\vspace{.3 cm}

\section{\bf Introduction}
Many Mathematical concepts can be represented by the notion of set theory, which dichotomize the situation into the conditions: either \textquotedblleft yes\textquotedblright~ or \textquotedblleft no\textquotedblright. Till 1965, Mathematicians were concerned only about \textquotedblleft well-defined\textquotedblright~ things, and smartly avoided any other possibility which are more realistic in nature. For instance the set of tall persons in a room, the set of hot days in a year etc. In the year 1965, Prof. L.A. Zadeh \cite{S} introduced fuzzy set to accommodate real life situations by giving partial membership to each element of a situation under consideration.    

Keeping in view that fuzzy set theory lacks the parametrization tool, Molodtsov \cite{S1} introduced soft set as another mathematical framework to  deal with real life situations. 
Then comes another generalization of sets, namely fuzzy soft set, which is a hybridization of fuzzy sets and soft sets, in which soft set is defined over fuzzy set. Similar generalization have also spread to topological space. The notion of topological space is defined on crisp sets and hence is affected by different generalizations of crisp sets like fuzzy sets and soft sets. C.L. Chang \cite{CA} introduced fuzzy topological space in 1968 and subsequently \c{C}a\v{g}man {\textit {et al}}. \cite{S4} and Shabir {\textit {et al}}. \cite{S2} introduced soft topological space independently in 2011. In the same year B. Tanay {\textit {et al}}. \cite{S5} introduced fuzzy soft topological spaces and studied neighborhood and interior of a fuzzy soft set and then used these to characterize fuzzy soft open sets. Recently Roy {\textit {et al}}. \cite{S7} have obtained different conditions for a subfamily of fuzzy soft sets to be a fuzzy soft basis or fuzzy soft subbasis. Levine \cite{NL} introduced the concepts of semi-open sets and semicontinuous mappings in topological spaces and were applied in the field of Digital Topology \cite{AR}. Azad \cite{AZ} initiated the study of these sets in fuzzy setting and in \cite{PKD}, authors carried the study in soft topological spaces.\\

This paper aims to generalize open and closed sets in fuzzy soft topological spaces as semiopen and semiclosed fuzzy soft sets. Then various set theoretic properties related to these generalized sets are to be studied. Further, it is intended to generalize the structures like interior and closure via semiopen and semiclosed fuzzy soft sets and study their properties.

It is presumed that the basic concepts like fuzzy sets, soft sets and fuzzy soft sets etc. are known to the readers. However below are some definitions and results required in the sequel.

\begin{defn} 
Let $f_E $ be a fuzzy soft set, $\mathcal{FS}(f_E)$ be the set of all fuzzy soft subsets of $f_E $, $\tau$ be a subfamily of $\mathcal{FS}(f_E)$ and $A, B, C \subseteq E$. Then $\tau$ is called a fuzzy soft topology on $f_E$ if the following conditions are satisfied
\begin{enumerate}[(i)]
\item $\overset{\sim}{\Phi}_E, f_E$ belongs to $\tau$;
\item $h_A, k_B \in \tau \Rightarrow h_A\overset{\sim}{\bigcap} k_B \in \tau $;
\item $\{(g_C)_\lambda~|~ \lambda \in \Lambda\} \subset \tau \Rightarrow \overset{\sim}{\underset{\lambda \in \Lambda}{\bigcup}} (g_C)_\lambda \in \tau $.
\end{enumerate}
Then $(f_E, \tau)$ is called a fuzzy soft topological space. Members of $\tau$ are called fuzzy soft open sets and their complements are called fuzzy soft closed sets.
\end{defn}

\begin{defn} \cite{JM}
A fuzzy soft set $g_A$ is said to be a fuzzy soft point, denoted by $\huge{e}_{g_{\scriptscriptstyle A}}$, if for the element $e \in A, g(e) \neq \overset{\sim}{\Phi}$ and $g(e^{'}) = \overset{\sim}{\Phi}, \forall e^{'} \in A-\{e\}$.
\end{defn}

\begin{defn} \cite{JM}
A fuzzy soft point $e_{g_{\scriptscriptstyle A}}$ is said to be in a fuzzy soft set $h_A$, denoted by $e_{g_{\scriptscriptstyle A} } \overset{\sim}{\in} h_A$ if for the element $e \in A, g(e) \leq h(e)$.
\end{defn}

 


\section{Semiopen and semiclosed fuzzy soft sets}

Generalization of closed and open sets in topological spaces are of recent advances. Here, we introduce semiopen and semiclosed fuzzy soft sets and study various set theoretic properties related to these structures. The concepts of closure and interior are generalized via semiopen and semiclosed fuzzy soft sets.

\begin{defn}
In a fuzzy soft topological space $(f_E, \tau)$, a fuzzy soft set 

\begin{enumerate}
 \item $g_ A$ is said to be semiopen fuzzy soft set if $\exists$ an open fuzzy soft set $h_A$ such that $h_A \overset{\sim }{\subseteq} g_A \overset{\sim }{\subseteq} cl(h_A)$;
\item $p_ A$ is said to be semiclosed fuzzy soft set if $\exists$ a closed fuzzy soft set $k_A$ such that $int(k_A) \overset{\sim }{\subseteq} p_A \overset{\sim }{\subseteq} k_A$;
\end{enumerate}

\end{defn}

\begin{exm}
Let $U = \{h^1, h^2, h^3\}$ and $E= \{e_1, e_2, e_3\}$. Consider a fuzzy soft set $f_E = \{(e_1, \{h^1_{0.2}, h^2_{0.8}, h^3_{0.5}\}), (e_2, \{h^1_{0.8}, \\h^2_{0.1}, h^3_{1}\}), (e_3, \{h^1_{0.7}, h^2_{0.5}, h^3_{0.2}\})\}$ defined on $U$. Then the subfamily

$\tau = \{\overset{\sim}{\Phi}_E, f_E, \{(e_1, \{h^1_{0.2}, h^2_{0.4}, h^3_{0.1}\})\}, \\
~~~~~~~~~~\{(e_1, \{h^1_{0.1}, h^2_{0.5}, h^3_{0.5}\}), (e_2, \{h^1_{0.7}, h^2_{0}, h^3_{0.7}\}), (e_3, \{h^1_{0.6}, h^2_{0.1}, h^3_{0.1}\})\}\\
~~~~~~~~~~\{(e_1, \{h^1_{0.2}, h^2_{0.6}, h^3_{0.4}\}), (e_2, \{h^1_{0.1}, h^2_{0.1}, h^3_{0.9}\}), (e_3, \{h^1_{0.5}, h^2_{0.5}, h^3_{0.1}\})\}\\
 ~~~~~~~~~~\{(e_1, \{h^1_{0}, h^2_{0.8}, h^3_{0.5}\}), (e_2, \{h^1_{0.8}, h^2_{0}, h^3_{0.1}\}), (e_3, \{h^1_{0.4}, h^2_{0.3}, h^3_{0}\})\} \\
 ~~~~~~~~~~\{(e_1, \{h^1_{0.2}, h^2_{0.8}, h^3_{0.5}\}), (e_2, \{h^1_{0.8}, h^2_{0.1}, h^3_{0.9}\}), (e_3, \{h^1_{0.5}, h^2_{0.3}, h^3_{0.1}\})\}\\
 ~~~~~~~~~~\{(e_1, \{h^1_{0.1}, h^2_{0.8}, h^3_{0.5}\}), (e_2, \{h^1_{0.8}, h^2_{0}, h^3_{0.7}\}), (e_3, \{h^1_{0.6}, h^2_{0.3}, h^3_{0.1}\})\}\\
~~~~~~~~~~\{(e_1, \{h^1_{0.2}, h^2_{0.5}, h^3_{0.5}\}), (e_2, \{h^1_{0.7}, h^2_{0}, h^3_{0.7}\}), (e_3, \{h^1_{0.6}, h^2_{0.1}, h^3_{0.1}\})\}\\
~~~~~~~~~~\{(e_1, \{h^1_{0.2}, h^2_{0.8}, h^3_{0.5}\}), (e_2, \{h^1_{0.8}, h^2_{0}, h^3_{0.1}\}), (e_3, \{h^1_{0.4}, h^2_{0.3}, h^3_{0}\})\}\\
~~~~~~~~~~\{(e_1, \{h^1_{0.2}, h^2_{0.6}, h^3_{0.5}\}), (e_2, \{h^1_{0.7}, h^2_{0.1}, h^3_{0.9}\}), (e_3, \{h^1_{0.6}, h^2_{0.5}, h^3_{0.1}\})\}\}$\\
 
is a fuzzy soft topology on $f_E$ and $(f_E, \tau)$ is a fuzzy soft topological space.\\

Here $g_E = \{(e_1, \{h^1_{0.1}, h^2_{0.4}, h^3_{0.5}\}), (e_2, \{h^1_{0.1}, h^2_{0}, h^3_{0.7}\}), (e_3, \{h^1_{0.5}, h^2_{0.1}, h^3_{0.1}\})\}$ is a semiopen fuzzy soft set.
\end{exm}

\begin{rem}
Every open (closed) fuzzy soft set is a semiopen (semiclosed) fuzzy soft set but not conversely. 
\end{rem}

\begin{rem}
$\overset{\sim}{\Phi}_E$ and $f_E$ are always semiclosed and semiopen.
\end{rem}

\begin{rem}
Every clopen set is both semiclosed and semiopen.
\end{rem}

From now onwards, we shall denote the family of all semiopen fuzzy soft sets (semiclosed fuzzy soft sets) of a fuzzy soft topological space $(f_E, \tau)$ by $SOFSS(f_E) $ ($SCFSS(f_E)$).

\begin{thm}
Arbitrary union of semiopen fuzzy soft sets is a semiopen fuzzy soft set.
\end{thm}
\begin{proof}
Let $\{(g_A)_\lambda~|~ \lambda \in \Lambda\}$ be a collection of semiopen fuzzy soft sets of a fuzzy soft topological space $(f_E, \tau)$. Then $\exists$ an open fuzzy soft sets $(h_A)_\lambda$ such that $(h_A)_\lambda \overset{\sim }{\subseteq} (g_A)_\lambda \overset{\sim }{\subseteq} cl((h_A)_\lambda)$ for each $\lambda$; hence $\overset{\sim }{\bigcup}(h_A)_\lambda \overset{\sim }{\subseteq} \overset{\sim }{\bigcup}(g_A)_\lambda \overset{\sim }{\subseteq}  cl(\overset{\sim }{\bigcup}(h_A)_\lambda)$ and $\overset{\sim }{\bigcup}(h_A)_\lambda$ is open fuzzy soft set. So $\overset{\sim }{\bigcup}(g_A)_\lambda$ is a semiopen fuzzy soft set.
\end{proof}

\begin{rem}
Arbitrary intersection of semiclosed fuzzy soft sets is a semiclosed fuzzy soft set.
\end{rem}
\begin{thm}
If a semiopen fuzzy soft set $g_A$ is such that $g_A \overset{\sim }{\subseteq} k_A \overset{\sim }{\subseteq} cl(g_A) $, then $k_A$ is also semiopen.
\end{thm}
\begin{proof}
As $g_A$ is semiopen fuzzy soft set $\exists$ an open fuzzy soft set $h_A$ such that $h_A \overset{\sim }{\subseteq} g_A \overset{\sim }{\subseteq} cl(h_A)$; then by hypothesis $h_A \overset{\sim }{\subseteq}  k_A$ and $cl(g_A) \overset{\sim }{\subseteq} cl(h_A) \Rightarrow k_A \overset{\sim }{\subseteq} cl(g_A) \overset{\sim }{\subseteq} cl(h_A) $ i.e., $h_A \overset{\sim }{\subseteq} k_A \overset{\sim }{\subseteq} cl(h_A)$, hence $k_A$ is a semiopen fuzzy soft set.
\end{proof}

\begin{rem}
It is not true that the intersection (union) of any two semiopen (semiclosed)  fuzzy soft sets need not be a semiopen (semiclosed) fuzzy soft set. Even the intersection (union) of a semiopen (semiclosed) fuzzy soft set with a fuzzy soft open (closed) set may fail to be a semiopen (semiclosed) fuzzy  soft set. It should be noted that in general topological space the intersection of a semiopen set with an open set is a semiopen set \cite{TN} but it doesn't hold in fuzzy setting \cite{AZ}. Further it should be noted that the closure of a fuzzy open set, is a fuzzy semiopen set and the interior of a fuzzy closed set is a fuzzy semiclosed set. 
\end{rem}

\begin{thm}
If a semiclosed fuzzy soft set $m_A$ is such that $int(m_A) \overset{\sim }{\subseteq} k_A \overset{\sim }{\subseteq} m_A $, then $k_A$ is also semiclosed.
\end{thm}

Following two theorems characterize semiopen and semiclosed fuzzy soft sets.
\begin{thm}
A fuzzy soft set $g_A \in SOFSS(f_E)  \Leftrightarrow$ for every fuzzy soft point $e_{g_{\scriptscriptstyle A}} \overset{\sim}{\in} g_A, \exists$ a fuzzy soft set $h_A \in SOFSS(f_E)$ such that $e_{g_{\scriptscriptstyle A}} \overset{\sim}{\in} h_A \overset{\sim }{\subseteq}g_A$.
\end{thm}

\begin{proof}

Take $h_A = g_A$, this shows that the condition is necessary.

For sufficiency, we have $g_A= \overset{\sim}{\underset{e_{g_{\scriptscriptstyle A}} \overset{\sim}{\in} g_A}{\bigcup}}(e_{g_{\scriptscriptstyle A}}) \overset{\sim}{\subseteq} \overset{\sim}{\underset{e_{g_{\scriptscriptstyle A}} \overset{\sim}{\in} g_A}{\bigcup}}h_A \overset{\sim}{\subseteq} g_A$.
\end{proof}

\begin{thm}
If $g_A$ is any fuzzy soft set in a fuzzy soft topological space $(f_E, \tau)$ then following are equivalent: 
\begin{enumerate}
\item $g_A$ is semiclosed fuzzy soft set;
\item $int(cl(g_A)) \overset{\sim }{\subseteq} g_A$;
\item $cl(int(g_A^c)) \overset{\sim }{\supseteq} g_A^c$.
\item $g_A^c$ is semiopen fuzzy soft set;

\end{enumerate}
\end{thm}

\begin{proof}

$(i) \Rightarrow (ii)$ If $g_A$ is semiclosed fuzzy soft set, then $\exists$ closed fuzzy soft set $h_A$ such that $int(h_A) \overset{\sim }{\subseteq} g_A \overset{\sim }{\subseteq} h_A \Rightarrow int(h_A) \overset{\sim }{\subseteq} g_A \overset{\sim }{\subseteq} cl(g_A) \overset{\sim }{\subseteq}h_A$. By the property of interior we then have $int(cl(g_A)) \overset{\sim }{\subseteq} int(h_A) \overset{\sim }{\subseteq} g_A$; 

$(ii) \Rightarrow (iii) int(cl(g_A)) \overset{\sim }{\subseteq} g_A \Rightarrow g_A^c \overset{\sim }{\subseteq} (int(cl(g_A)))^c = cl(int(g_A^c)) \overset{\sim}{\supseteq} g_A^c$.

$(iii) \Rightarrow (iv)$ $h_A = int(g_A^c)$ is an open fuzzy soft set such that $int(g_A^c) \overset{\sim }{\subseteq} g_A^c 
\overset{\sim }{\subseteq} cl(int(g_A^c)) $, hence $g_A^c$ is semiopen.

$(iv) \Rightarrow (i)$ As $g_A^c$ is semiopen $\exists$ an open fuzzy soft set $h_A$ such that $h_A \overset{\sim }{\subseteq} g_A^c \overset{\sim }{\subseteq} cl(h_A) \Rightarrow h_A^c$ is a closed fuzzy soft set such that $g_A \overset{\sim }{\subseteq} h_A^c$ and  $g_A^c \overset{\sim }{\subseteq} cl(h_A) \Rightarrow int(h_A^c) \overset{\sim }{\subseteq} g_A$, hence $g_A$ is semiclosed fuzzy soft set. 
\end{proof}

\begin{defn}
Let $(f_E, \tau)$ be a fuzzy soft topological space and $g_A$ be a fuzzy soft set over $U$.
\begin{enumerate}
\item The fuzzy soft semi closure of $g_A$ is a fuzzy soft set $fssclg_A= \overset{\sim }{\bigcap} \{s_A~|~ g_A \overset{\sim }{\subseteq} s_A$ and $ s_A \in SCFSS(f_E) \}$;
\item The fuzzy soft semi interior of $g_A$ is a fuzzy soft set $fssintg_A= \overset{\sim }{\bigcup} \{s_A~|~ s_A \overset{\sim }{\subseteq} g_A$ and $ s_A \in SOFSS(f_E)\}$.
\end{enumerate}
\end{defn}

$fssclg_A$ is the smallest semiclosed fuzzy soft set containing $g_A$ and $fssintg_A$ is the largest semiopen fuzzy soft set contained in $g_A$.

\begin{thm}
Let $(f_E, \tau)$ be a fuzzy soft topological space and $g_A$ and $k_A$ be two fuzzy soft sets over $U$, then 
\begin{enumerate}
\item $ g_A \in SCFSS(f_E)  \Leftrightarrow g_A= fssclg_A$;
\item $ g_A \in SOFSS(f_E) \Leftrightarrow g_A= fssintg_A$;
\item $(fssclg_A)^c = fssint(g_A^c)$;
\item $(fssintg_A)^c = fsscl(g_A^c)$;
\item $g_A \overset{\sim }{\subseteq} k_A \Rightarrow fssintg_A \overset{\sim }{\subseteq} fssintk_A$;
\item $g_A \overset{\sim }{\subseteq} k_A \Rightarrow fssclg_A \overset{\sim }{\subseteq} fssclk_A$;
\item $fsscl \overset{\sim}{\Phi}_E = \overset{\sim}{\Phi}_E$ and $fsscl f_E = f _E$;
\item $fssint \overset{\sim}{\Phi}_E = \overset{\sim}{\Phi}_E$ and $fssint f_E = f _E$;
\item $fsscl(g_A \overset{\sim }{\cup} k_A) = fssclg_A \overset{\sim }{\cup}  fssclk_A$;
\item $fssint(g_A \overset{\sim }{\cap} k_A) = fssintg_A \overset{\sim }{\cap}  fssintk_A$;
\item $fsscl(g_A \overset{\sim }{\cap} k_A) \overset{\sim }{\subset} fssclg_A \overset{\sim }{\cap}  fssclk_A$;
\item $fssint(g_A \overset{\sim }{\cup} k_A) \overset{\sim }{\subset} fssintg_A \overset{\sim }{\cup}  fssintk_A$;
\item $fsscl(fssclg_A)=fssclg_A$;
\item $fssint(fssintg_A)=fssintg_A$.
\end{enumerate}
\end{thm}

\begin{proof} Let $g_A$ and $k_A$ be two fuzzy soft sets over $U$.
\begin{enumerate}
\item Let $g_A$ be a semiclosed fuzzy soft set. Then it is the smallest semiclosed set containing itself and hence\\ $g_A= fssclg_A$. 

On the other hand, let $g_A= fssclg_A$ and $fssclg_A \in SCFSS(f_E) \Rightarrow g_A \in SCFSS(f_E)$.
\item Similar to $(i)$.
\item \begin{eqnarray*}
(fssclg_A)^c & = & (\overset{\sim }{\bigcap} \{s_A~|~ g_A \overset{\sim }{\subseteq} s_A and  s_A \in SCFSS(f_E) \})^c \\
&=& \overset{\sim }{\bigcup} \{s_A^c~|~ g_A \overset{\sim }{\subseteq} s_A and  s_A \in SCFSS(f_E) \} \\
&= &\overset{\sim }{\bigcup} \{s_A^c~|~ s_A^c \overset{\sim }{\subseteq} g_A^c and  s_A^c \in SOFSS(f_E)\} \\ 
&=& fssint(g_A^c).
\end{eqnarray*}
\item Similar to $(iii)$.
\item Follows from definiton.
\item Follows from definition.
\item Since $\overset{\sim}{\Phi}_E $ and $f_E$ are semiclosed fuzzy soft sets so $fsscl \overset{\sim}{\Phi}_E  = \overset{\sim}{\Phi}_E $ and $fsscl f_E = f _E$.  
\item Since $\overset{\sim}{\Phi}_E $ and $f_E$ are semiopen fuzzy soft sets so $fssint \overset{\sim}{\Phi}_E  = \overset{\sim}{\Phi}_E $ and $fssint f _E = f _E$. 
\item We have $g_A \overset{\sim }{\subset} g_A \overset{\sim }{\bigcup} k_A$ and  $k_A \overset{\sim }{\subset} g_A \overset{\sim }{\bigcup} k_A$. Then by $(vi), fssclg_A \overset{\sim }{\subset} fsscl(g_A \overset{\sim }{\bigcup} k_A)$ and $fssclk_A \overset{\sim }{\subset} fsscl(g_A \overset{\sim }{\bigcup} k_A) \Rightarrow fssclk_A\overset{\sim }{\bigcup} fssclg_A \overset{\sim }{\subset} fsscl(g_A \overset{\sim }{\bigcup} k_A)$. 

Now, $fssclg_A, fssclk_A \in SCFSS(f_E) \Rightarrow fssclg_A \overset{\sim }{\bigcup} fssclk_A \in SCFSS(f_E)$.

Then $g_A \overset{\sim }{\subset} fssclg_A$ and $k_A \overset{\sim }{\subset} fssclk_A$ imply $g_A\overset{\sim }{\bigcup}k_A \overset{\sim }{\subset} fssclg_A \overset{\sim }{\bigcup}fssclk_A$.i.e., $fssclg_A \overset{\sim }{\bigcup}fssclk_A$ is a semiclosed set containing $g_A\overset{\sim }{\bigcup}k_A$. But $fsscl(g_A \overset{\sim }{\bigcup}k_A)$ is the smallest semiclosed fuzzy soft set containing $g_A\overset{\sim }{\bigcup}k_A$. Hence $fsscl(g_A \overset{\sim }{\bigcup}k_A) \overset{\sim }{\subset} fssclg_A \overset{\sim }{\bigcup}fssclk_A$. 
 So, $fsscl(g_A \overset{\sim }{\cup} k_A) = fssclg_A \overset{\sim }{\cup}  fssclk_A$.
\item Similar to $(ix)$.
\item We have $g_A \overset{\sim }{\bigcap} k_A \overset{\sim }{\subset} g_A$ and $g_A \overset{\sim }{\bigcap} k_A \overset{\sim }{\subset} k_A\\ \Rightarrow fsscl(g_A \overset{\sim }{\bigcap} k_A) \overset{\sim }{\subset} fssclg_A$ and $fsscl(g_A \overset{\sim }{\bigcap} k_A) \overset{\sim }{\subset} fssclk_A \\ \Rightarrow fsscl(g_A \overset{\sim }{\bigcap} k_A) \overset{\sim }{\subset} fssclg_A\overset{\sim }{\bigcap} fssclk_A$. 
\item Similar to $(xi)$.
\item Since $fssclg_A \in SCSS(U)$ so by $(i), fsscl(fssclg_A)=fssclg_A$.
\item Since $fssintg_A \in SOSS(U)$ so by $(ii), fssint(fssintg_A)=fssintg_A$.
\end{enumerate}
\end{proof}

\begin{rem}
If $g_A$ is semiopen fuzzy soft (semiclosed fuzzy soft) set, then int($g_A$), fssint($g_A$) (fsscl($g_A$) and cl($g_A$)) are semiopen fuzzy soft
(semiclosed fuzzy soft) set.
\end{rem}

\section{Conclusion}
In this work, we have initiated the generalization of closed and open sets in a fuzzy soft topological space as semiopen and semiclosed fuzzy soft sets. We have also discussed some characterizations of these sets. Further the topological structures namely interior and closure are also generalized and several interesting properties are studied. Several remarks are stated which give comparison between the properties of these sets in three different domains, namely general topology, fuzzy topology and fuzzy soft topology. Surely the discussions in this paper will help researchers to enhance and promote the study on fuzzy soft topology for its applications in practical life.


\section{References}
\begin{thebibliography}{99}
\bibitem{AR}
A. Rosenfeld, Digital topology, The American Mathematical Monthly, 86(8), (1979) 621–630, 1979.
\bibitem{S5}
B. Tanay, M. Burc Kandemir, Topological structure of fuzzy soft sets, Comp. Math. Appl., 61, (2011), 2952--2957.

\bibitem{CA}
C.L. Chang, Fuzzy topological spaces, J. Math. Anal. Appl., 24, (1968), 182--190.

\bibitem{S1}
D. Molodtsov, Soft Set Theory-First Results, Comp. Math. Appl. 37, (1999) 19-31.

\bibitem{PKD} 
J. Mahanta, P.K. Das, On soft topological space via semiopen and semiclosed soft sets, arXiv:1203.4133v1 [math.GN] 16 Mar 2012.
\bibitem{JM} 
J. Mahanta, P.K. Das, Results on fuzzy soft topological spaces, arXiv:1203.0634v1 [cs.IT] 3 Mar 2012.
\bibitem{AZ} 
K. K. Azad, On fuzzy semicontinuity, fuzzy almost continuity and fuzzy weakly continuity, Journal of Mathematical Analysis and Applications,82(1), (1981) 14–32.
\bibitem{S} 
L.A. Zadeh, Fuzzy Sets, Information and Control, 11, (1965), 341--356.

\bibitem{S2}
M. Shabir, M. Naz, On fuzzy soft topological spaces, Comp. Math. Appl., 61, (2011) 1786-1799.

\bibitem{S4}
N. Cagman, S. Karatas, S, Enginoglu, Soft Topology, Comp. Math. Appl., 62, (2011) 351-358. 

\bibitem{NL}
N. Levine, "Semi-open sets and semi-continuity in topological spaces," The American Mathematical Monthly, 70, (1963) 36–41.

\bibitem{S7}
S. Roy, T.K. Samanta, A note on fuzzy soft topological spaces, Ann. Fuzzy Math. Inform., 3(2), (2012), 305--311.
\bibitem{TN}
T. Noiri, On semi-continuous mapping, Atti Accud. Naz. Lincei Rend, Cl. Sci. Fis. Mat. Nutur (8) 54 (1973), 132-136.
\end{thebibliography}
\end{document}